\documentclass{amsart}

\usepackage[mathscr]{eucal}
\usepackage{amsmath,amssymb}
\usepackage{graphics}
\usepackage{multirow}
\usepackage{hyperref}
\usepackage[active]{srcltx}

\newtheorem{thm}{Theorem}[section]
\newtheorem{THM}{Theorem}
\newtheorem{cor}[thm]{Corollary}
\newtheorem{prop}[thm]{Proposition}

\newtheorem{lemma}[thm]{Lemma}

\theoremstyle{definition}

\newtheorem{remark}[thm]{Remark}

\DeclareMathOperator{\Hom}{Hom}

\DeclareMathOperator{\rank}{rank}

\DeclareMathOperator{\ad}{ad}

\def\Q{\mathbb Q}
\def\R{\mathbb R}
\def\C{\mathbb C}

\def\P{\mathbb P}
\def\p{\mathfrak p}
\def\Z{\mathbb Z}
\def\F{\mathcal F}
\def\E{\mathcal E}
\def\D{\mathcal D}
\def\G{\mathcal G}
\def\H{\mathcal H}
\def\K{\mathcal K}
\def\L{\mathcal L}

\input xy
\xyoption{all}

\begin{document}

\title[Foliations with vanishing Chern classes]
{Foliations  with vanishing Chern classes}
\author[ J.V. Pereira and F. Touzet ]
{ Jorge Vit\'{o}rio PEREIRA$^{1}$ and Fr\'ed\'eric TOUZET$^2$}
\address{\newline $^1$ IMPA, Estrada Dona Castorina, 110, Horto, Rio de Janeiro,
Brasil \hfill\break $^2$ IRMAR, Campus de Beaulieu, 35042 Rennes Cedex, France
}\email{$^1$ jvp@impa.br}
\email{$^2$ frederic.touzet@univ-rennes1.fr}


\maketitle

\begin{abstract}
In this paper  we aim at the description of   foliations  having tangent sheaf $T\mathcal F$ with    $c_1(T\mathcal F)=c_2(T\mathcal F)=0$ on non-uniruled projective manifolds.
We prove that the universal covering of the ambient manifold splits as a product, and that the Zariski closure  of a general  leaf of $\mathcal F$ is an  Abelian variety.
It turns out that the analytic type of the Zariski closures of leaves
may vary from leaf to leaf.  We discuss how this variation is related to arithmetic properties of the tangent sheaf of the foliation.
\end{abstract}
\setcounter{tocdepth}{1}
\tableofcontents \sloppy

\section{Introduction and statement of results}\label{S:intro}

If $X$ is a compact K\"ahler manifold (e.g. $X$ is a projective manifold)  with real Chern classes satisfying $c_1(TX)=c_2(TX)=0$ then   Yau's solution to Calabi conjecture combined with a result by Apte implies that $X$ admits a flat Hermitian metric. One can thus evoke a classical result by Bierberbach to conclude that  there exists a finite \'etale morphism from a complex torus to  $X$, see \cite[Corollary 4.15]{kob2} and references  therein.

In this paper we aim at a  generalization of this result  where we replace the hypothesis on the tangent bundle of $X$ by the same hypothesis on saturated coherent subsheaves of the tangent bundle of $X$. We will also assume that $X$ is a non-uniruled, i.e. there is no rational curve passing through a general point of $X$, and projective. Even if many of our arguments do work  in the more general context of compact K\"ahler manifolds, at multiple places we will have to restrict to  projective manifolds.

Let $X$ be a complex manifold. A distribution $\D$ on $X$ is determined by a coherent subsheaf $T\D$ of $TX$ (the tangent sheaf of $\D$) which  has torsion free cokernel $TX/T\D$, in other words $T\D$ is saturated in $TX$. The generic rank of $T\D$ is the dimension of the distribution.  A foliation $\F$ on $X$ is a distribution with involutive (i.e. local sections are closed under Lie bracket)  tangent sheaf $T\F$.

\subsection{Previous results}
Before stating our results we recall some  partial answers to the problem of classifying distributions with vanishing Chern classes which will be useful in and/or motivate what follows .

\subsubsection{Foliation with trivial tangent sheaf}
If $\D$ is a distribution of positive dimension with trivial tangent bundle on a compact complex manifold $X$  then $h^0(X,TX)\ge h^0(X,T\D)= \dim \D>0$ and it follows that $X$ admits holomorphic vector fields.  Below we state a  theorem of Lieberman \cite{li},  see also \cite[Theorem 3.2]{AMN}, which completely describes the situation when the ambient is projective and non-uniruled.

\begin{thm}\label{T:lieb}
If $X$ is non-uniruled projective manifold with $h^0(X,TX)>0$ then up to a finite \'{e}tale covering,   $X=A\times Y$  where $A$ is an Abelian variety and $Y$ satisfies $h^0( Y,TY)=0$.
\end{thm}

It follows that the distribution $\D$ is smooth,  involutive and the underlying foliation  is the pull-back under the  natural projection to $T$ of a linear  foliation.

\subsubsection{Foliations with trivial canonical class}\label{croco3}
In \cite{croco3}, a joint work with F. Loray, we have shown   that distributions on non-uniruled projective manifolds satisfying $c_1(T\D)=0$ are smooth and involutive.  Therefore,  if we restrict ourselves to the category of non-uniruled projective manifolds then our problem about distributions reduces to a problem about smooth foliations.
Under the same hypothesis, we have also proved in \cite{croco3}  the existence of a transverse smooth foliation which together with $\D$ provide a splitting of the tangent bundle of $X$;  and that the determinant of $T\D$ is a torsion line-bundle. These results, with precise statements, are recalled in \S \ref{S:polystable}.

As corollary of the statement concerning $\det T\D$ one obtains that a foliation $\F$ of dimension one on a non-uniruled manifolds with $c_1(T\F)=0$ is defined by a global holomorphic vector on a suitable finite \'etale covering. Hence, we can apply Lieberman's result quoted above
to conclude that $\F$ is tangent to an isotrivial  fibration by abelian varieties.

\subsubsection{Codimension one  foliations}
Smooth foliations of codimension one with $c_1(T\F)=0$ on compact K\"ahler manifolds have been classified in \cite{Touzet}. In particular, when $c_2(T\F)=0$ we have the following possibilities:
\begin{enumerate}
\item up to a finite \'etale covering $X$ is a complex torus and $\F$ is a linear foliation on it; or
\item up to a finite \'etale covering $X$ is the product of a complex torus and a curve and $\F$ is the pull-back of a linear foliation on the torus by the natural projection; or
\item $X$ is  a $\P^1$-bundle over a complex torus, and $\F$ is everywhere transverse to the fibers of this $\P^1$-bundle.
\end{enumerate}
In all cases $\mbox{det}\ (T\F)$ is a torsion line-bundle, and when $X$ is not uniruled then after a finite \'etale covering $X$ splits as the product of a complex torus and a smooth manifold of dimension $0$ or $1$.

We will proceed to describe the new results proved in this paper. The remaining    of this introduction   reflects the structure of the paper with each subsection describing the content of the  corresponding section of the paper.

\subsection{Splitting of the universal covering}
Let $X$ be a projective non-uniruled projective manifold and $\F$ be a foliation on $X$ with $c_1(T\F)=0$.
Using results from (\cite{croco3}) we prove  that
$T\F$ is polystable whenever $c_1(T\mathcal F)=0$.
Therefore $T\mathcal F$ is indeed an Hermite-Einstein bundle  by  a theorem of Donaldson. Specializing to foliations which satisfy the additional assumption $c_2(T\mathcal F)=0$, we obtain that $T\mathcal F$ is a flat hermitian bundle  and as such  carries a flat connection with unitary monodromy
\[
   \rho:\pi_1(X)\rightarrow U(r,\mathbb C)\subset GL(r,\mathbb C)
\]
It turns out that this representation is also the monodromy representation of a transversely Euclidean foliation everywhere transverse to $\F$. Exploiting the transverse geometry of this foliation, similarly to what we have done in a previous joint work with
M. Brunella \cite{BPT},  we prove the following result.

\begin{THM}\label{TI:A}
Let $\F$ be a  foliation with $c_1(T\mathcal F)=c_2(T\mathcal F)=0$ in $H^*(X,\mathbb R)$ on a projective manifold manifold  $X$. If  $X$ is not uniruled then $\F$ is a smooth foliation, there exists a smooth foliation $\F^\perp$ which together with $\F$ induces a splitting $TX=T\F \oplus T\F^\perp$  of the tangent bundle of $X$,  and the universal covering of $X$ splits as a product  $\C^{\dim(\F)} \times Y$.
\end{THM}

The projectivity of $X$ is used only to prove the polystability of $T\F$, and the proof  is based on the pseudo-effectiveness of $KX$. If we assume that $X$ is a compact K\"ahler manifold with pseudo-effective $KX$ then the same conclusion should probably holds true. Also, if we drop the non-uniruledness assumption but replace it with the existence of decomposition of the tangent bundle then the proof of Theorem \ref{TI:A} can be adapted to
prove the following result.

\begin{THM}\label{TI:Abis}
Let $X$ be a compact K\"ahler manifold such that the tangent bundle splits
as a direct sum of two subbundles $A\oplus B$. If $A$ is an involutive subbundle of $TX$ and admits a  flat hermitian metric then the universal covering of $X$ splits as a product $\C^{\rank(A)}\times Y$ compatible with the splitting of $TX$.
\end{THM}

This provides further evidence to Beauville's conjecture concerning the universal covering of compact K\"ahler manifolds with split tangent bundle, see \cite{Beauville, BPT, Horing} and references therein.

\subsection{Shafarevich map and structure theorem}
In view of Lieberman's result, one might expect that in the presence of a foliation with vanishing Chern classes  there is no need to pass to universal covering to obtain a splitting of the ambient manifold: a finite covering would suffice. It turns out that the situation is more delicate
and the existence of such splitting is determined the holonomy representation
$\rho : \pi_1(X) \to U(r, \C)$ of the  hermitian flat bundle $T\F$.

If the representation $\rho$ has finite image then, after a finite \'etale covering, we obtain a foliation trivial tangent sheaf and we are reduced to Lieberman's Theorem. Otherwise, if the image of $\rho$ is infinite then a result of Zuo on the Shafarevich map of representations \cite{Z}  allows us to prove the following structure theorem.

\begin{THM}\label{TI:B}
Let $\F$ be a  foliation on a projective manifold $X$. If   $c_1(T\mathcal F)=c_2(T\mathcal F)=0$ in $H^*(X,\mathbb R)$
then, after passing to a finite \'{e}tale covering,   there exists a meromorphic fibration  on $X$  whose general fiber $F$ is an abelian variety,  the foliation $\mathcal F$ is
tangent to this fibration, and the restriction of $\F$ to   $F$ is  a linear foliation.
\end{THM}

By a meromorphic fibration we mean a rational map $f: X \dashrightarrow  B$ whose restriction to an open subset $X^0 \subset X$ is a proper morphism over an open subset  $B^0$ of the base.

\subsection{Infinite monodromy}
Perhaps it is worth noticing at this point that the image of $\rho$ can be indeed  infinite. This later property implies the non isotriviality of the abelian fibration and  a construction carried out by Faltings  in \cite{Faltings} provides examples. In \S \ref{S:Faltings} we exhibit foliations of dimension two and codimension three with flat tangent bundle with infinite monodromy. The Abelian fibration given by Theorem \ref{TI:B} is smooth with fibers of dimension four, and therefore the linear  foliations have codimension two in the fibers.

\subsection{Codimension two} Studying the variation of Hodge structures determined by the Abelian fibration, we are able to prove that the above mentioned examples are optimal: they have  minimal dimension and minimal  codimension among the examples with infinite representation $\rho$.

\begin{THM}\label{TI:C}
Let $\F$ be a  foliation of codimension two  on a non-uniruled  projective manifold $X$. If   $c_1(T\mathcal F)=c_2(T\mathcal F)=0$ in $H^*(X,\mathbb R)$
then, after passing to a finite \'{e}tale covering,     $X = A  \times Y$ where $A$ is an Abelian variety and $Y$ is a point, a smooth curve, or  a surface. The foliation $\F$ is the pull-back of a linear foliation on $A$ under the projection to the first factor.
\end{THM}

\subsection{Arithmetic}
Another possible approach to prove our structure Theorem is using reduction to positive characteristic. As the leaves of our foliation are uniformized by Euclidean spaces they are Liouvillian in the sense of pluripotential theory and one might hope to be able to use Bost's Theorem \cite{Bost}. Although we are able to prove the existence of a non-trivial foliation $\G$, $p$-closed for almost every prime $p$, containing our foliation with vanishing Chern classes, Faltings example shows that $\G$ does not have necessarily vanishing Chern classes and we are unable to control the universal coverings of its leaves.
Nevertheless, we can prove the following statement.

\begin{THM}\label{TI:D}
Let $\F$ be a  foliation on complex projective manifold $X$ both defined over a finitely generated $\Z$-algebra $R$. Suppose that $\F$ is maximal, with respect to inclusion, among the foliations with   $c_1(T\F)=0$ and  $c_2(T\F)=0$. Then at least one of the following assertions holds true.
\begin{enumerate}
\item Up to a finite \' etale covering,  $X$ is isomorphic to a product of an Abelian variety $A$ with another projective manifold $Y$, and $\F$ is  the pull-back of  a linear foliation on $A$ under the natural projection $A\times Y \to A$.
\item For a dense set of maximal primes $\p$ in $Spec(R)$  the reduction modulo $\p$ of $T \F$ is not  Frobenius semi-stable. Moreover,  there exists a non-empty open subset $U\subset Spec(R)$ such that for every maximal prime $\p \in U$ the reduction modulo $\p$ of $\F$ is either $p$-closed or the reduction modulo $\p$ of $T\F$ is not Frobenius semi-stable.
\end{enumerate}
 \end{THM}

In particular, the foliations  presented in Section \ref{S:Faltings} have tangent sheaf which are  stable but not strongly semi-stable for infinitely many primes. To the best of our knowledge the only previously known examples of this phenomena appeared in  \cite{Brenner}.

\subsection*{Acknowledgements} We are grateful to Jo\~ao Pedro dos Santos for bringing to our knowledge the references \cite{Brenner} and \cite{Langer}.

\section{Polystability  and splitting}

\subsection{Polystability of the tangent sheaf}\label{S:polystable}
Let $\mathcal E$ be a coherent sheaf  on a  $n$-dimensional smooth projective  variety $X$ polarized
by an ample line bundle $H$. The slope of $\mathcal E$ (more precisely the $H$-slope of $\mathcal E$)  is defined as the quotient
\[
\mu ( \mathcal E) = \frac{ c_1( \mathcal E) \cdot H^{n-1} }{ \rank ( \mathcal E) } \, .
\]
If the slope of every coherent  proper subsheaf $\mathcal E'$ of $\mathcal E$ satisfies $\mu ( \mathcal E') < \mu ( \mathcal E )$ (respectively $\mu ( \mathcal E') \le \mu ( \mathcal E )$) then $\mathcal E$ is called stable (respectively semi-stable).

A vector bundle $\mathcal E$ is said to be polystable if it can be expressed as a finite sum
$$\mathcal E=\bigoplus_i {\mathcal E}_i$$
where each summand ${\mathcal E}_i$ is a stable subbundle.

The goal of this paragraph is to prove the following lemma.

\begin{lemma}\label{L:polystable}
If $\D$ is a distribution with $c_1(T\D)=0$ on a non-uniruled projective manifold $X$ then $\D$ is smooth, integrable, and has polystable tangent sheaf.
To wit, there exists a finite family ${\mathcal F}_1,...,{ \mathcal F}_p$  of smooth holomorphic subfoliations of $\D$  whose tangent sheafs are stable with respect to any given polarisation,  have zero first Chern class,  and satisfy
$$T\mathcal D=\bigoplus_{i=1}^p{T\mathcal F}_i.$$
\end{lemma}

Most of the arguments that will be used in the proof of this lemma already appeared in \cite{croco3}. We now proceed to  briefly recall them.

On the one hand, if $X$ is  not uniruled then Boucksom-Demailly-Paun-Peternell characterization of uniruledness \cite{BDPP} implies that the canonical bundle of $X$ is pseudo-effective. On the other hand, if   $c_1(T\mathcal D)=0$ then
the $T\D$ can  defined as the kernel of a holomorphic $q$-form with coefficients in the
bundle $KX\otimes \det(T\D)$. Demailly's Theorem implies that $\D$ is integrable, i.e. $\D$ is not only a distribution but is also a  foliation.

The smoothness of $\D$ is proved in \cite{croco3}. As the relevant result will be essential in the proof of Lemma \ref{L:polystable} we  reproduce here its statement.

\begin{thm}\label{T:croco}
Let  $X$ be a projective manifold with $KX$ pseudo-effective and $L$ be a pseudo-effective line bundle on $X$.
If $v \in H^0(X,\bigwedge^p TX\otimes L^*)$ is a non-zero section then the zero set of $v$  is empty.
Moreover, if  $\mathcal D$ is a codimension $q$ distribution on  $X$ with $c_1(T\mathcal D)=0$  then $\mathcal D$
is a  smooth foliation (i.e. $T \mathcal D$ is involutive)
and there exists another  smooth holomorphic foliation
$\mathcal G$ of dimension $q$ on $X$ such that $ T X = T \mathcal D \oplus T \mathcal G$.
\end{thm}

\noindent{\bf Conclusion of the proof of Lemma \ref{L:polystable}}
We already know that  $\D$ is smooth and integrable. To remind us of the integrability of $\D$ let us denoted it by $\F$ instead.
Campana-Peternell in \cite{CPT}  proved that  the canonical sheaf of every saturated coherent subsheaf of $TX$ is pseudo-effective. Since $c_1(T\F)=0$,  it follows that  $T\mathcal F$ is semi-stable bundle with respect to any polarization of $X$.

Assume now that $T\mathcal F$ is not stable. Then there exists a distribution ${\D}_0$ tangent to $\F$, in other words  $T{{\D}_0}$ is a subbundle of $T\F$,  such that $c_1(T\D_0)=0$. Since $\D_0$ satisfies the same  hypothesis of $\D$, we get that it is integrable and we can  apply Theorem \ref{T:croco} to deduce
that is is smooth, and to
exhibit another smooth foliation $\D_0^{\perp}$ such that  $TX= T\D_0 \oplus  T\D_0^{\perp}$. If we set  ${\F_0}^\perp$ as the foliation obtained as the intersection of $\D$ and $\D_0^{\perp}$, i.e.
$T{\F_0}^\perp:=T \D_0^{\perp} \cap T\mathcal F$, then
$$
T{\F}=T\D_0\oplus T{\D_0}^\perp
$$
and, consequently,  $c_1(T{\D_0}^\perp)=0$. The Lemma follows by induction. \qed

\subsection{Maximal foliations with vanishing Chern classes}
Lemma \ref{L:polystable} ensures the existence of maximal foliation with vanishing Chern classes.

\begin{cor}
If $\F$ and $\G$ are foliations on a non-uniruled projective manifold which verify  $c_1(T\F)=c_1(T\G)=0$ then there exists a foliation $\H$ containing $\F$ and $\G$ and with $c_1(T\H)=0$. Moreover, if $c_2(T\F)=c_2(T\G)=0$ then we can choose $\H$ with $c_1(T\H)=c_2(T\H)=0$.
\end{cor}
\begin{proof}
Consider the morphism $\varphi: T\F \oplus T\G \to TX$ sending $(v,w)$ to $v+w$. Since
$X$ is not uniruled then the image of $\varphi$ has non-positive degree with respect to any polarization.
The polystability of $T\F$ and $T\G$ implies that the kernel of $\varphi$ has degree zero and is a polystable summand  of both sheaves. On the one hand we have that the image of $\varphi$ is locally free; and on the other hand both the kernel and the image of $\varphi$ have  vanishing Chern classes. Theorem \ref{T:croco} implies the image of $\varphi$ is an involutive
subbundle of $TX$  with Chern polynomial $c(T\F) \cdot c(T\G) \cdot ( c(\ker \varphi)^ {-1})$,  and is the tangent sheaf of the sought foliation.
\end{proof}

\subsection{Hermite-Einstein structure on the tangent sheaf}
Let us recall a  result by  Donaldson \cite{Donaldson}.

\begin{thm}
Let $E$ a polystable holomorphic vector bundle over a projective manifold. Then $E$ carries an Hermite-Einstein metric. In particular, if $E$ satisfies $c_1(E)=c_2(E)=0$ then this metric is flat.
\end{thm}

We can apply Donaldson's Theorem combined with Lemma \ref{L:polystable} to deduce the following corollary.

\begin{cor}\label{C:flathermit}
If $\F$ is a foliation on  projective non uniruled manifold $X$ such that $c_1(T\mathcal F)=0$, then there exists  an Hermite-Einstein metric on $T\F$. In particular,  $T\F$  is flat hermitian whenever $c_2(T\mathcal F)=0$.
\end{cor}

If $\F$ is a foliation satisfying the hypothesis of Theorem \ref{TI:A} then
$T\F$ is a flat unitary  vector bundle and as such is defined by an unitary representation
$$\rho:\pi_1(X)\rightarrow U(r,\C).$$

The next result, also from  \cite{croco3},  implies that the
induced representation $\det(\rho) : \pi_1(X) \to U(1,\C)$
has finite image.

\begin{prop}\label{P:torsion}
If $\D$ be a distribution with $c_1(T\D)=0$  on a non-uniruled projecive manifold  then its canonical bundle $K\D={\det{T\D}^*}$ is a  torsion line bundle.
\end{prop}

As a corollary we obtain  the following result.

\begin{cor}\label{C:finitemonodromy} If the image $\mbox{Im}\ \rho$ of $\rho$ is virtually solvable (i.e contains a solvable group of finite index), then $Im\  \rho$ is finite.
\end{cor}
\begin{proof}One can find a finite etale covering $e:X^e\rightarrow X$ such that $\mbox{Im}\ e^*\rho$ is solvable and contained in a connected solvable linear algebraic group. This representation is semi-simple (being unitary), hence splits as a sum of one dimensional representations.
The finiteness of $\mbox{Im}\  \rho$ follows from  Proposition \ref{P:torsion}.
\end{proof}

\subsection{Proof of Theorem \ref{TI:A}}
Let $\F$ be a holomorphic foliation on a non uniruled projective manifold $X$ with $c_1(T\F)=c_2(T\F)=0$.
Let us call $\F^\perp$ the (non necessarily unique) complementary foliation of $\F$ whose existence is ensured by  Theorem \ref{T:croco},
 and consider the universal covering projection  $\pi:\tilde X\rightarrow X$.

The foliation ${\F}^\perp$ is defined by a holomorphic  $1$-form $\omega$ with values in $T\mathcal F$ . This latter being flat hermitian (Corollary \ref{C:flathermit}), we get $\nabla\omega=0$ using the  K\"ahler identities (here, $\nabla$ denotes the unitary flat connection attached to $E=T\mathcal F$). In particular,  ${\F}^\perp$ is transversely Euclidean; moreover, this transversal hermitian structure is complete by compactness of $X$. Take a primitive $F=(F_1,..., F_r)$ of $\pi^*\omega$ on $\tilde X$. The restriction of $F$ to any leaf of the foliation $\pi^*\F$  is a covering map onto ${\mathbb C}^r$, hence an isomorphism by simple connectedness. Moreover, any such leaf is equipped with the Euclidean metric induced by $\sum {|dF_i|}^2$.

On the other hand, the leaves of $\pi^* {\F}^\perp$ must coincide with connected components of the  levels of $F$. We claim that for every $x\in{\C}^r$, $F^{-1}(x)$ is connected. Indeed, assume that  $F^{-1}(x)$ is a union of distinct leaves of $\pi^* {\F}^\perp$, say $\L_j$ for $j$ belonging to some   set $J$.
For each $j\in J$ denote by $V_j$ the saturation of $\L_j$ by $\pi^* {\F}$. As the intersection of any leaf of  $\pi^* \F$  with $F^{-1}(x)$ consists in exactly one point ($F$ restricted to the leaves of $\pi^* \F$ is a covering map) it follows  that the sets $V_j$, $j \in J$, are pairwise  disjoint open subsets of $\tilde X$ such that $\tilde X= \cup_{j \in J} V_j$. Connectedness of $\tilde X$ implies that $J$ has cardinality one.

Now, fix a leaf $\L\simeq{\C}^r$ of $\pi^*\F$ and a leaf $\L^\perp$ of $\pi^*\F^\perp$ . For every $a\in\tilde X$, consider the leaf $\L_a$ of $\pi^*\F$ and the leaf ${\L_a^\perp}$ such that $\{a\}=\L_a\cap{\L_a^\perp}$. We get a well defined bijection
$$\tilde X\rightarrow \L\times\L^\perp$$
sending $a$ to $(a_{\L},a_{\L^\perp})$ such that $\{a_\L\}={\L_a^\perp}\cap \L$ and $\{a_{\L_a^\perp}\}={\L}_a\cap \L^\perp$. This bijection is a biholomorphism by a  standard flow-box argument, and we get the expected trivialization $\tilde X\simeq{\mathbb C}^r\times Y$. \qed

\medskip

As the reader can easily verify the above argument also proves Theorem \ref{TI:Abis}.

\section{Structure Theorem}

\subsection{Shafarevich map}
Let $X$ be a smooth projective algebraic variety and $\rho:\pi_1(X)\rightarrow G$ a representation of the fundamental group. One can define the Shafarevich map for $\rho$ as follows: it is a surjective rational morphism with connected fibers
  $$\mbox{sh}_\rho:X\rightarrow \mbox{Sh}_\rho (X)$$
where  $\mbox{Sh}_\rho (X)$ is a normal algebraic variety such that for any irreducible subvariety $V\subset X$ not contained in a union of countably many proper algebraic subvarieties, $\mbox{sh}_\rho (V)=$ point iff $\rho (\pi_1(V))$ is finite.
It is easy to see that the existence of $Sh_\rho$ is unique up to birational equivalence.

Koll\'ar (\cite{K}) has proved that $\mbox{Sh}_\rho$ always exists with the additional property that it is a proper morphism restricted to some  Zariski open set of $X$. These constructions have been extended to compact K\"ahler manifolds by Campana (\cite{c}).

\subsection{Zuo's theorem}
Before stating Zuo's Theorem \cite{Z} let us recall some  definitions from the theory of algebraic groups.
A connected algebraic group $G$ is called  {\it almost simple} if it is non commutative and all its proper algebraic normal subgroups are finite.
If  $S$ is a semi-simple connected algebraic group then a classical result asserts that $S$ is isogenous to a  product of almost simple algebraic groups.

\begin{thm}\label{T:zuo}
Let $\rho\colon \pi_1(X)\to G$ be a Zariski dense representation into an almost simple algebraic group. Then there exists a finite \'etale covering $e\colon X^e\to M$ such that the pull-back representation $e^*\rho$ factors through the Shafarevich map ${\rm sh}_{e^*\!\rho}\colon X^e\to {\rm Sh} _{e^*\!\rho}(X^e)$ and the Shafarevich variety ${\rm Sh}_{e^*\rho}(X^e)$ is projective algebraic of general type.
\end{thm}

\subsection{Tangency to the fibers of the  Shafarevich map}
Recall that $T\mathcal F$ is flat hermitian, hence comes from a representation $\rho:\pi_1(X)\rightarrow U(r)$.
From now on, let us deal with the Zariski closure $G$ of $\mbox{Im}\  \rho \subset GL(r,\mathbb C)$. After taking a finite covering of $X$, one can assume that $\mbox{Im}\ \rho$ is torsion free and $G$ is connected.

Write $G$ as $R\rtimes S$ where $R$ is the solvable radical  of $G$ and $S$ is a semi-simple group. The latter  group decomposes as
 $$ S=(S_1\times S_2\times....\times S_p)/H$$
where the $S_i$ are quasi-simple and $H$ is a finite subgroup of $S_1\times S_2\times....\times S_p$.

Let $H_i$ be the image of $H$ under the projection of $S_1\times...\times S_p$ onto the $i^{\mbox{th}}$ factor.

Now, projecting $\rho$ to the almost simple group $G_i:=S_i/(H_i)$ in the semisimple factor, we obtain a Zariski dense representation
$$\rho_i:\pi_1(X)\rightarrow G_i.$$
Consider the Shafarevich map ${\rm sh}_{e^*\!\rho_i}\colon X^e\to {\rm Sh} _{e^*\!\rho_i}(X^e)$ as  in  Theorem  \ref{T:zuo} (note that the finite etale covering $X^e$ may depend on $i$).

\begin{prop}\label{Ftangent}
The pull-back foliation $e^*\mathcal F$ is tangent to the fibers of ${\rm sh}_{e^*\!\rho_i}$. Consequently,
 the foliation $\mathcal F$ is tangent to the fibers of ${\rm sh}_{\!\rho_i}$.
\end{prop}
\begin{proof}
 Assume that the statement of this proposition is false. We use here the splitting $\tilde X={\mathbb C }^r\times Y$ given by theorem \ref{TI:A}. There exists an euclidean subspace ${\mathbb C}^p,\ 1\leq p\leq r$ of the factor ${\mathbb C}^r$ and a local analytic connected subspace $Z_Y$ of $Y$ such that the natural  meromorphic map ${\mathbb C }^p\times Z_Y$ induced by the covering projection $\tilde X\rightarrow X^e$  and ${\rm sh}_{e^*\!\rho_i}$ is a local biholomorphism near some point; this is absurd because ${\rm Sh} _{e^*\!\rho_i}(X^e)$ is measure hyperbolic (being of general type) whereas ${\mathbb C }^p\times Z_Y$ is not. See \cite[Chapter 7]{kob}  for the related properties of  hyperbolicity.
\end{proof}

\begin{prop}
The foliation $\F$ is tangent to the fibers of ${\rm sh}_{\!\rho}$.
\end{prop}
\begin{proof}
Denote by $\rho_S$ the representation $\pi_1(M)\rightarrow S\simeq G/R$ induced by $\rho$. Using Proposition \ref{Ftangent}, one can see that $\F$ is   tangent to the fibers of ${\rm sh}_{\!\rho_S}$. Let $U$ be a Zariski open subset of $X$ such that ${\rm sh}_{\!\rho}$ is a smooth proper fibration on $U$ and pick a fiber   $F_\rho$ of ${{\rm sh}_{\!\rho}}_{|U}$. The restriction of $\rho$ to $F_\rho$ takes values into a virtually solvable group. The image of $\rho_{|F_\rho}$ is then finite by corollary \ref{C:finitemonodromy}. Therefore  ${\rm sh}_{\!\rho}$ and  ${\rm sh}_{\!\rho_S}$ coincide.
\end{proof}

\subsection{Proof of Theorem \ref{TI:B}}
After replacing $X$ by $X^e$  we can choose an open Zariski subset $U$ of $X$ such that  $\Phi_U:={{\rm sh}_{\!\rho}}_{|U}$ is a smooth proper fibration onto its image $V$ with the additional property that  $h^0(TF)$ is the same for every fiber $F$  (semi-continuity). In particular, one has $h^0(T F)\geq h^0(T\F_{|F})$.
Since $X$ in non-uniruled the same holds true for every fiber $F$ of $\Phi_U$ by a result of Fujiki \cite{Fujiki} (stability of uniruledness), and we can apply Theorem \ref{T:lieb}
to deduce that the fibers are foliated by Abelian varieties.

If $i: U \to X$ denotes the inclusion and $TU / V$ the relative tangent sheaf of the fibration then
$i_* {{\rm sh}_{\!\rho}}^* {{\rm sh}_{\!\rho}}_*   T U  / V$ maps to a subsheaf of $TX$ which after saturation becomes the tangent sheaf of  a foliation $\mathcal A$. The general leaves of $\mathcal A$ are Abelian varieties contained in  the fibers   of ${{\rm sh}_{\!\rho}}$, and containing $\mathcal F$ as a sublinear foliation. We consider the map to the Hilbert scheme which associates to a point $x \in X$ the point in ${\rm Hilb}$ corresponding to the leaf of
$\mathcal A$ through $x$, for details see  \cite{GM}. As $\mathcal A$ is smooth on $U$ this morphism will give rise to the sought meromorphic fibration. \qed

\begin{remark}
 This description shows that the representation $\rho$ arises from a variation of polarized  Hodge structures, hence takes values in a number field.
\end{remark}





\section{Infinite monodromy}\label{S:Faltings}

This Section is devoted to presenting examples of foliations with Chern classes on non-uniruled projective manifolds with infinite monodromy representation.
In most of it  we follow very closely the presentation of \cite[Section 5]{Faltings}.

\subsection{Quaternion algebras}

Let $K$ be a  field of characteristic zero. Let $A$ and $B$  be two elements
of $K$ and denote by  $D=D(A,B)$ the associated quaternion algebra. Concretely,
$D$ is the non-commutative  $K$-algebra with underlying $K$-vector space
generated by $1,i,j,k$ and subject to relations:
\[
i^2=A, j^2=B, ij=k,  \text{ and } ji=-k .
\]
It follows in particular that $jk=i, ki=j,$ and $k^2 = -AB$.

The algebra $D$ carries a canonical involution which takes $\alpha=a + bi + cj + dk$ to
$ \alpha^* = a - bi - cj - dk$, and consequently $N(\alpha) = \alpha \cdot {\alpha}^*$ (the norm of $\alpha$) and $T(\alpha) = \alpha +  \alpha^*$ (the trace of $\alpha$) belong to $K$.

Notice that $N(a + bi + cj + dk) = a^2 - A b^2 - B c^2 +AB d^2$ and therefore is a quadratic form on $D$. If we extend the scalars to $\overline K$, the algebraic closure of $K$, then the result is   isomorphic to the algebra of $2\times 2$ matrices with coefficients
in $\overline K$, i.e.
$D \otimes_K \overline K  \simeq M_2(\overline K)$. Moreover, such isomorphism can be chosen in such a way that the norm and the trace are respectively identified with the determinant and the trace of matrices.

If there exits a non-zero  element in $D$ with zero norm then $D$ is isomorphic to $M_2(K)$, otherwise $D$ is a division algebra with the left inverse of an element $\alpha$ given by $- \alpha^* / N(\alpha)$.

If $K$ is a real  number field and  we choose an embedding  $\sigma: K \to \mathbb R$ then two things can happen: the real quadratic form on $D \otimes_{\sigma(K)} \mathbb R$ induced by $N$ is positive definite and then $D \otimes_{\sigma(K)} \mathbb R$ is isomorphic to the algebra of quaternions $\mathbb H$; or  the real quadratic form is indefinite and  $D \otimes_{\sigma(K)} \mathbb R$ is isomorphic to $M_2(\R)$.

\subsection{Example}
From now on $K$ will be a real quadratic number field and we will choose $A,B \in K$
such that for one of the embeddings  we get $D \otimes_{\sigma_1(K)} \mathbb R \simeq \mathbb H \begin{large}\begin{Large}                                                                                                           \end{Large}                                                                                              \end{large}$ while for the other we get $D \otimes_{\sigma_2(K)} \mathbb R\simeq M_2(\R)$. In other words $D\otimes_{\mathbb Q} \mathbb R \simeq  \mathbb H \times M_2(\R)$. Notice that in this case $D$ is a division $K$-algebra, the non-invertible elements appear only after extension of scalars.

Let $V_{\mathbb Q}$ be the $8$-dimensional $\Q$-vector space underlying $D$, i.e. $V_{\Q}=D$. Of course, $D$ acts on $V$ by left and right multiplication. If $\tau$ is an element of $D$ such that $\tau^* = - \tau$ then it defines a new involution on $D$: $\alpha \mapsto a^{\tau} = \tau  \alpha^* \tau^{-1}$. If we consider the skew-symmetric on $V$ defined by
\[
< \alpha, \beta > = < \alpha,\beta>_{\tau} = \mathrm{tr}_{K/\Q}(T(\alpha \tau  \beta^*))
\]
then $<\alpha \gamma, \beta > = < \alpha, \beta \gamma^{\tau}>$. Moreover the action by left mulplication of the algebraic group $G$ with rational points
\[
G(\Q)= \{ \alpha \in D \, | \, N(a) =1 \},
\]
preserves $< , >$.

The group $G(\R)$ of real points of  $G$ is isomorphic to $SU(2,\C) \times SL(2,\R)$ and
it acts on $V_{\R}=V\otimes_\Q \R\simeq \C^2 \oplus ( \R^2 \otimes \R^2)$ in such way that $SU(2,\C)$ acts on the first summand by its natural representation and $SL(2,\R)$ acts on the the first factor of the tensor product also by its natural representation.
The bilinear form $< , >$ decomposes as
\[
< ,> = < , > _1 \oplus <, >_2 \otimes <,>_3,
\]
where $<,>_1$ in suitable coordinates is a real multiple of the skew-form $(z,w)\mapsto Im(\overline z_1 w_1 + \overline z_2 w_2 )$ on $\C^2$, and $<,>_2$ is a real multiple of the skew-form $(x,y)\mapsto x_1y_2 - x_2y_1$ on $\R^2$. We will choose $\sigma$ in such a way that $<,>_1$ is a positive multiple of the skew-form above and that $< ,>_3$ is positive definite.

If $W_0$ is the kernel of the natural multiplication morphism $\C^2 \otimes_{\R} \C\to \C$ ($W_0 = \{ z\otimes 1 + i z \otimes i | z \in \C^2\}$) and $W_+ = \{ z \otimes 1 - iz \otimes i | z \in \C^2\}$  then
\[
V_{\mathbb C} = V_{\Q} \otimes_{\Q}\C =V_{\R}\otimes_{\R} \C= (W_+ \oplus W_0) \oplus (\C^2 \otimes \R^2)\,.
\]
Notice that $W_0$ and $W_+$ are orthogonal with respect to the complexification of $< ,>_1$ and that they are interchanged by complex conjugation.
If $z \in W_+$ then $-i^{-1} < z, \overline z>_1> 0$. Let now $z=(z_1,z_2) \in \C^2=\R^ 2\otimes_{\R} \C$ (the first factor of the tensor product  in the decomposition above) be such that $-( i )^{-1} < z,\overline z>_2>0$ and consider the line complex line $L \subset \C^ 2$ determined by it. Therefore $<,>$ defines    a weight one polarized Hodge structure on $V$ with
\[
V_{\C}^{1,0} = (W_+ \oplus 0) \oplus (L \otimes \R^2)  .
\]
The stabilizer in $G(\R)$ of this Hodge structure is a maximal compact subgroup $M \simeq SU(2,\C) \times U(1,\C)$, and $G(\R)$ acts transitively on the set of polarized Hodge structures of $V$ with polarization given by $<,>$. Thus they are parametrized by $G(\R)/ M$ which is isomorphic to the Poincar\'e disc $\mathbb D$.

If we choose a lattice $V_{\Z} \subset V_{\Q}$ over which $<,>$ takes integral values and of determinant one then we get a family of $4$-dimensional simple Abelian varieties parametrized by the Poincar\' e disk.  If $\Gamma \subset G(\Q)$ is a torsion-free arithmetic subgroup which stabilizes $V_{\Z}$ (and, as $\Gamma$ cannot be contained in $M$, does not preserve the Hodge structure) then we obtain a family $f:  X \to B$ over the compact Riemann surface
$B = \Gamma \backslash \mathbb D = \Gamma \backslash  G(\R) / M$. The subspace $(W_+ \oplus 0) \oplus 0 \subset V_{\C}$  gives rise to a rank two local system $\mathbb W$ over $B$ contained in $(R^1 f_* \C)^ {1,0}$,   and with  monodromy given by the image of  $\Gamma \subset G(\Q) \subset G(\R) \simeq SU(2,\C)\times SL(2,\R)$ under the projection to $SU(2,\C)$. Therefore $\mathbb W \otimes  \mathcal O_ B$ can be seen as a subsheaf of $f_*\Omega^1_{X/B}$.  Koll\'ar's decomposition theorem (\cite{K2}),
implies that $f_*\Omega^1_{X/B}$ splits holomorphically as a direct sum of subbundles
    $$f_*\Omega^1_{X/B}=(\mathbb W \otimes  \mathcal O_ B)\oplus \mathcal L$$ with $\mathcal L$ ample on $B$ (in our case, $\mathcal L$ corresponds to the factor $L\otimes{\R}^2$). It follows the existence of  a canonical dual splitting
 $$f_*T_{X/B}={(\mathbb W \otimes  \mathcal O_ B)}^\perp\oplus {\mathcal L}^\perp \, .$$
The  factor  $f^*({\L}^\perp )$ determines a rank two subsheaf of $T_{X/B} \subset TX$. The corresponding foliation has trivial  Chern classes and is tangent to the fibers of $f$. Over a point $x \in B$ the Abelian fourfold $f^ {-1}(x)$ is given by the dual of
\[
    \frac{(W_+ \oplus 0) \oplus (L_x \otimes \R^ 2)}{\Lambda} \, ,
\]
 where $\Lambda$ is the projection
of $V_{\Z}$ to the first factor of the decomposition $V_{\C} =  V^{1,0}_{\mathbb C,x} \oplus \overline{ V^{1,0}_{\mathbb C,x} }$. The foliation on $f^ {-1}(x)$  is the linear foliation  determined by the kernel of $L_x \otimes \R^ 2\subset V^{1,0}_{\mathbb C,x}$.

This example shows that the unitary representation attached to the tangent sheaf of a foliation with vanishing Chern classes  can be indeed infinite. It also shows that the hypothesis on the codimension of the foliation in \cite[Theorem 1.5]{TouTou} is necessary, contrarily to what
was conjectured there.

\section{Codimension two}

In this section we will prove  Theorem \ref{TI:C} by analyzing the variation of weight one polarized Hodge structures attached
to the meromorphic fibration given by Theorem \ref{TI:B}.

\subsection{Settling the notation}

Let  $\F$ be  a codimension two foliation with $c_1(T\F)=c_2(T\F)=0$ on a non-uniruled projective manifold $X$. We will denote by $n$ the dimension of $X$ and by $r$ the dimension of $\F$.
Let
$$\rho:\pi_1(X)\rightarrow U(r,\C)$$
be the representation attached to $T\F$, and recall that there exists a transverse foliation ${\F}^{\perp}$ defined by a closed $1$-form with values in the flat vector bundle $T\F$.

Then, by Theorem \ref{TI:B},  we can assume that $\F$ is tangent to the fibers of a meromorphic fibration $\Phi:X\dashrightarrow Y$  with general  fiber being an Abelian variety. Thus  there exists open subsets $U \subset X$ and $V\subset Y$ such that $\Phi_{|U}$
is a proper smooth fibration with connected fibers
 over $V\subset Y$.
With the additional data of an ample divisor on $X$, this defines a weight one polarized variation of Hodge structure over $V$.
Let $F$ be a fiber of $\Phi$ and denote by $q_F$ the polarization form on $F$.
Call $\psi$ the natural representation
$$\psi: \pi_1(V)\rightarrow \mbox{Aut}\ H^1(F,\C)$$ associated to the local system $R^1\Phi_*{\C}_U$ and consider $H^1(F,\C)$ as a $\pi_1(V)$-module.
Let $G$ be the image of $\pi_1(V)$ by $\psi$.

 Let us define $N^{1,0}$ as the maximal $\pi_1(V)$-submodule of $H^1(F,\C)$  contained in $H^{1,0}(F)$. Because $\pi_1(V)$ acts isometrically with respect to the scalar product
$$N^{1,0}\times N^{1,0} \ni (\omega_1,\omega_2) \mapsto {q_F}(\omega_1,\overline{\omega_2}),$$ 
it is indeed a unitary submodule.
The orthogonal (non necessarily unitary) submodule to $N=N^{1,0}\oplus \overline {N^{1,0}}$  with respect to $q_F$ is of the form
   $$B=B^{1,0}\oplus\overline{B^{1,0}}$$ where $B^{1,0}$ is a complementary subspace of $N^{1,0}$ in $H^{1,0}$.

\subsection{Proof of Theorem \ref{TI:C}}
 Notice that $H^0(F, i^*{(N{\F}^\perp)}^*)$ is a submodule of $N^{1,0}$ (here, $i$ denotes the inclusion of $F$ into $V$) and that the image of the induced action of $\pi_1(V)$ is precisely $\mbox{Im}\ \rho$.
Assume for a moment that $N^{1,0}=H^{1,0}(F,\C)$, then $H^1(F,\C)$ is a unitary module. As $G$ preserves the integer lattice $H^1(F,\Z)$, it is a finite group by a theorem of Kronecker.

From now on, we are going to deal with $\mbox{Im}\ \rho$ infinite, $\mbox{codim}\ \F=2$ and aim at a contradiction.
By the previous observations, we obtain that $N^{1,0}= H^0(F, i^*{(N{\F}^\perp)}^*)$ and that $B= B^{1,0} \oplus \overline{B^{1,0}}$  has dimension $2$.
The group of automorphisms $\mbox{Gal}(\C/\Q)$ acts  naturally on the set $\mathcal M$ of $\pi_1(V)$-submodules of $H^1(F,\C)$
since the action of $\pi_1(V)$ preserves $H^1(F,\Q)$ and therefore is defined over $\Q$. For $M\in\mathcal M$ we will denote $M^\sigma$ its conjugate by $\sigma\in\mbox{Gal}(\C/\Q)$.
There exists finitely many irreducible unitary submodules $N_1,....,N_l$ such that
       $$N^{1,0}=\bigoplus_i N_i$$
corresponding to a splitting of $\rho=\rho_1\oplus...\oplus \rho_l$.
Recall that $\mbox{Im}\ \rho_i$ is finite whenever $\mbox{dim}_\C\ N_i=1$.
Hence, there exists some indices $i$ such that $\mbox{dim}_C\ N_i\geq 2$.
Let us call $p_N$, respectively $p_B$, the projection of $H^1(F,\C)$ to $N$, resp. to $B$.
We will distinguish three cases:
\begin{enumerate}
 \item For every $\sigma\in\mbox{Gal}(\C/\Q)$, ${(N^{1,0})}^\sigma\subset N$.
This means that the submodule $M$ of $N$ generated by   ${(N^{1,0})}^\sigma$ with $\sigma$ ranging in $\mbox{Gal}(\C/\Q)$ is defined over $\Q$, i.e. $M=W_\Q\otimes\C$ where $W_\Q$ is a subspace of $H^1(F,\Q)$. Thus it can be defined over $\Z$: $M=W_\Z\otimes\C$ with $W_\Z=W_\Q\cap H^1(F,\Z)$. Moreover, $M$ is a unitary module (being a submodule of the unitary module $N$) containing $N^{1,0}$. This implies that $\mbox{Im}\ \rho$ is finite,  a contradiction.
\item  There exists an irreducible factor $N_{i_0}$ of dimension at least $2$ and $\sigma\in\mbox{Gal}(\C/\Q)$  such that $p_N({(N_{i_0})}^\sigma)$ and $p_B({(N_{i_0})}^\sigma)$ are not $\{0\}$. By irreducibility, these images are irreducible submodules of $N$ and $B$ both isomorphic to ${(N_{i_0})}^\sigma$. In particular, the second projection is the whole $B$. As $p_N({(N_{i_0})}^\sigma)$ is unitary, the same holds true for $B$. One can then conclude that $G$ lies in a unitary group and again that $G$ and $\mbox{Im}\ \rho$ are finite, absurd.
\item  There exists  an irreducible factor $N_{i_0}$, $\mbox{dim}\ N_{i_0}\geq 2$ and $\sigma\in\mbox{Gal}(\C/\Q)$  such that ${(N_{i_0})}^\sigma=B$.
Since $q_F$ is defined over $\Q$, we have  that $q_F({\omega_1}^\sigma,{\omega_2}^\sigma)=\sigma (q_F(\omega_1,\omega_2))$. Here, the contradiction follows from the fact that $q_F$ is trivial on $N_{i_0}\times N_{i_0}$ whereas it is not on ${(N_{i_0})}^\sigma\times {(N_{i_0})}^\sigma=B\times B$.
\end{enumerate}
Since at least one of the three possibilities above always holds true, the Theorem follows. \qed

\section{Arithmetic}

In this section we analyze the behavior of foliations with vanishing Chern classes under reduction modulo primes. The foliations, varieties, and sheaves defined over a field of characteristic $p>0$ will be marked with a subscript $p$, or $\p$.  For more details about the reduction modulo primes of foliations see \cite{ets,croco3,SB} and references therein.

\subsection{Power map}

Let $k$ be an algebraically closed field of characteristic $p>0$, and let $S_p = \mathrm{Spec}(k)$.
In this section $X_p$ will be a smooth irreducible projective $S_p$-scheme.

A foliation $\mathcal F_p$ on $X_p$ is determined by a coherent subsheaf $T\mathcal F_p$ of $TX_p$ which is involutive
and has torsion free cokernel $TX_p/T\mathcal F_p$. Unlike in characteristic zero, where Frobenius integrability theorem
implies that a foliation   at a formal neighborhood of a general point is nothing but a fibration, a foliation does not need
to have a leaf through a general point. This is the case only when $T\mathcal F_p$ is not just involutive but also $p$-closed, i.e.,  closed under
$p$-th powers.

The $p$-closedness of $T\mathcal F_p$ is equivalent to the vanishing of the  morphism
of $\mathcal O_{X_p}$-modules
\begin{align*}
\times^p : Frob^* T \mathcal F_p & \longrightarrow \frac{TX_p}{T\mathcal F_p} \\
 a \otimes v &\longmapsto a v^p \, .
\end{align*}
Here $Frob : X_p \to X_p $ denotes the absolute Frobenius morphism. Thus $Frob$ is the identity  over the topological space, but the morphism
at the level of structural sheaves given by
\begin{align*}
 Frob^{ \sharp} : \mathcal O_{X_p} &\longrightarrow \mathcal O_{X_p} \\
f &\longmapsto f^p \, .
\end{align*}
Therefore $Frob^* T \mathcal F_p = \mathcal O_{X_p} \otimes_{Frob^{-1} \mathcal O_{X_p}} Frob^{-1} T\mathcal F_p \simeq  \mathcal O_{X_p} \otimes_{Frob^{-1} \mathcal O_{X_p}}  T\mathcal F_p   $ is isomorphic to $T \mathcal F_p$ as a sheaf of abelian groups, but not as a sheaf of $\mathcal O_{X_p}$-modules: in it $a^p\otimes v = 1 \otimes a v$.

\subsection{Canonical connection}

If $\mathcal E_p$ is an arbitrary coherent sheaf over $X$ then $Frob^* \mathcal E_p$ comes equipped with a canonical connection $\nabla : Frob^* \mathcal E_p \to \Omega^1_{X_p} \otimes Frob^* \mathcal E_p$, defined
as $\nabla(f \otimes v) = df \otimes v$. Clearly the sheaf of flat sections is a sheaf of ${\mathcal O_{X_p}}^p$-modules and generates $Frob^* \mathcal E_p$
as a sheaf of $\mathcal O_{X_p}$-modules.

If $\mathcal G_p$ is a subsheaf of $Frob^* \mathcal E_p$ then it is natural to inquire if there exists a $\mathcal H_p \subset \mathcal E_p$ such that $\mathcal G_p = Frob^* \mathcal H_p$.
 Such $\mathcal H_p$ exists if and only if the $\mathcal O_{X_p}$-morphism
\begin{align*}
\mathcal G_p &\longrightarrow \Omega^1_{X_p} \otimes \frac{Frob^* \mathcal E_p }{\mathcal G_p} \\
\sigma &\longmapsto \nabla(\sigma) \mod \Omega^1_{X_p} \otimes \mathcal G_p \, .
\end{align*}
induced by the canonical connection $\nabla$ is identically zero, cf. \cite[Section 2]{Langer}. When $\mathcal H_p$ does not exist the morphism above induces a non-trivial
$\mathcal O_{X_p}$-morphism
\begin{align*}
\Phi_{\mathcal G_p} : T X_p &\longrightarrow \Hom\left( \mathcal G_p, \frac{Frob^* \mathcal E_p }{\mathcal G_p} \right) \\
v &\longmapsto \left( \sum f_i \otimes \sigma_i \mapsto \sum df_i (v) \sigma_i  \mod \mathcal G_p\right)  .
\end{align*}

\begin{prop} \label{P:automorfismos}
Let $\mathcal F_p$ be a foliation on $X_p$. If  $\mathcal G_p \subset Frob^* T \mathcal F_p$ is the kernel of $\times^p$ then every germ of infinitesimal automorphism of $\mathcal F_p$ is
 contained in $\ker \Phi_{\mathcal G_p}$. In particular, $\ker \Phi_{\mathcal G_p}$ contains the smallest $p$-closed subsheaf of $TX_p$ containing $T\mathcal F_p$.
\end{prop}
\begin{proof}
Let $v$ be a germ of infinitesimal automorphism of  $\mathcal F_p$, i.e., $v \in TX_p(U)$ and $[v, T\mathcal F_p(U)] \subset T \mathcal F_p(U)$ for some non-empty  open
subset $U \subset X_p$. If $\sum f_i \otimes v_i \in Frob_{X_p}^* T\mathcal F_p (U) $ is an element in the kernel of $\times^p$ then
\[
\sum f_i v_i^p = 0 \mod T \mathcal F_p(U) \, .
\]
Since $v$ is an automorphism we have that $[v,\sum f_i v_i^p] = 0 \mod T \mathcal F_p(U)$. From
$[v,\sum f_i v_i^p] = \sum df_i(v) v_i^p + \sum f_i [v,v_i^p] = \sum df_i(v)v_i^p  - \sum f_i \ad_{v_i^p } (v) \, $  and $\ad_{v_i^p }(v) = (\ad_{v_i})^p (v)$
\cite[eq. (60), page 186]{Jacobson}  we deduce that $\sum f_i \ad_{v_i^p } (v)$ belongs to $T\mathcal F_p (U)$ and consequently
\[
\sum df_i(v) v_i^p = 0 \mod T \mathcal F_p(U) \, .
\]
This last identity  proves the first statement.  For the second statement it suffices to notice $v^{p^n}$ ($n\ge 1$)  are infinitesimal automorphisms of $\mathcal F_p$ for any local section $v \in T\mathcal F_p(U)$ and that these vector fields generate the $p$-closure of $T\mathcal F_p(U)$.
\end{proof}

\subsection{Lifting the $p$-envelope of a foliation}
Up to the end of this Section,  $\p$ will denote a maximal prime in a finitely generated $\Z$-algebra  $R$ with residue field of characteristic $p>0$.

\begin{prop}\label{P:saturation}
Let $\F$ be a semi-stable foliation on a polarized complex projective manifold $(X,H)$ satisfying $\deg(T\F):=\det (T\F) \cdot H^{n-1} =0$. If everything in sight is defined over a finitely generated $\Z$-algebra $R$
then  one of the following assertions hold true:
\begin{enumerate}
\item the foliation $\F_{\mathfrak{p}}$ is $p$-closed for all maximal primes $\p$ in an non-empty open  subset of $Spec(R)$;
\item the foliation $\F$ is tangent to a  foliation $\G$ with $\dim \G > \dim \F$ and $\det (T \G) \cdot H^{n-1} \ge 0$; or
\item  the foliation $\F$ is tangent to a  foliation $\G$ with $\dim X > \dim \G > \dim \F$, $\det (T \G) \cdot H^{n-1} <0$, and  the reduction modulo $\mathfrak p$ of  $T\F$ is not Frobenius  semi-stable for all maximal primes $\p$ in a dense   subset of $Spec(R)$. Moreover, there exists a non-empty open subset $U\subset Spec(R)$ such that for every maximal prime $\mathfrak p \in U$ the foliation $\F_{\p}$ is $p$-closed or $Frob^*T \F_{\p}$ is not semi-stable.
\end{enumerate}
\end{prop}

In case (2) we do not exclude the possibility of $\G$ being the foliation on $X$ with only one leaf, i.e.  $T \mathcal G = T X$, but this
happens only if $X$ is uniruled (when $KX$ is not pseudo-effective by \cite{BDPP}) or when $X$ has trivial canonical class.

\begin{proof}[{\bf Proof of Proposition \ref{P:saturation}}] Suppose that $\F_\p$  the reduction modulo $\p$ of $\F$ is not $p$-closed. Then the morphism $\times^p: Frob^* T \F_\p \to TX_\p / T \F_\p$ is non-zero. Let $ \mathcal I_\p$ be its image.

If $Frob^* T\F_\p$ is semi-stable then $\deg(\mathcal I_\p)\ge 0$. Otherwise, according to  \cite[Corollary $2^p$]{SB}, there exists a constant $C\ge 0$, independent of $\p$, such that $ \deg(\mathcal I_\p)\ge -C$.

Let $\mathcal E_\p$ be the saturation in $TX_\p$ of the inverse image of $\mathcal I_\p$ under the natural quotient morphism $TX_\p \to TX_\p / T \F_\p$.
Notice that $\mathcal E_\p$ contains $T\F_\p$ and, as $\mathcal I_\p$,  has degree bounded from below  by $-C$. Since $p$-th powers of vector fields tangent to $\F_\p$ give rise to infinitesimal automorphisms of $\F_\p$ it follows that $\E_\p$ is  involutive. Also, the cokernel of inclusion of $\E_\p$ in $TX_\p$ is torsion free and, when $\p$ varies, have degree uniformly bounded from above by $C$. Thus the family of involutive sheaves $\E_\p$ containing $T\F_\p$  belong to a bounded family,  \cite[Corollaire 2.3]{Grothendieck}, and there exists a foliation $\G$ in characteristic zero strictly containing $\F$ with a tangent sheaf that reduces modulo $\p$ to $\mathcal E_p$.

If $Frob^* T\F_\p$ is semi-stable  for  a dense   subset of maximal primes $\mathfrak p \in Spec(R)$ then $\deg(T\G) \ge 0$ and there is nothing else to prove. If instead $Frob^* T\F_\p$
is not semi-stable for  a dense   subset of maximal primes $\mathfrak p \in Spec(R)$  then   $\deg(T\G) < 0$ is not excluded. Aiming at a contradiction, let us  assume  that $T\G$ coincides with $TX$. Then for infinitely many primes
the sheaf  $\E_\p$ defined above coincides with $TX_\p$ and the kernel $\K_\p$ of $\times^p: Frob^* T \F_p \to TX_p / T \F_p$
has positive degree.  Since $\E_\p$ is generated, at a general point of $X_\p$, by $p$-th powers of $T\F_\p$ we can apply  Proposition \ref{P:automorfismos} to deduce that   $\K_\p$ is a Frobenius pull-back of a subsheaf of $T\F_\p$ contradicting the semi-stability of $T\F_\p$.
\end{proof}

\subsection{Proof of Theorem \ref{TI:D}} Let   $\F$ be a foliation on a complex projective manifold $X$ maximal among the foliations with tangent bundle with vanishing first and second Chern classes. At a first moment let us assume that $R$ is contained in a number field.

We will start by excluding case (2) of Proposition \ref{P:saturation}.  Indeed, if $\G$
is the foliation containing $\F$ given by Proposition \ref{P:saturation}
then since $T\G$ is pseudo-effective by \cite{CPT} and the polarization
in  Proposition \ref{P:saturation} is arbitrary, it follows that $c_1(T\G)=0$.
Theorem \ref{T:croco} ensures the existence of a transverse  foliation $\G^{\perp}$, and  $\G$ is tangent to the meromorphic fibration $\pi: X \dashrightarrow S$ given by Theorem \ref{TI:B}. Restricting to an open subset $U\subset X$ where the fibration is proper and  smooth we see that the conormal bundle of $\G^{\perp}$  has an injective natural morphism to   the relative cotangent bundle of the fibration by Abelian varieties. Let $N^* \G^{\perp}_{/S}$ denote its image
and  consider its direct image $(\pi_U)_*N^* \G^{\perp}_{/S}$. If we take a section of $N^* \G^{\perp}$ restricted to a fiber $F$ of $pi$, thus a closed holomorphic $1$-form,  then the parallel transport along the leaves of $\mathcal G^{\perp}$ provides canonical  extensions to a neighborhood of $F$ on $X$ which is still closed. Thus  $(\pi_U)_*N^* \G^{\perp}_{/S}$ is flat for the Gauss-connection. It follows that $N^*\G^{\perp}$ is hermitian flat on $U$. Since the complement of $U$ has no irreducible components invariant by $\G^{\perp}$ this hermitian flat structure on $N^*\G^{\perp}$ extends to the whole $X$.  Hence not only $c_1(T\G)=0$ but also $c_2(T\G)=0$,
 contradicting the maximality of $\F$.
Thus either $\F$ is $p$-closed for almost every prime of $R$ or $T\F_\p$ is not Frobenius semi-stable for infinitely many primes $\p$.

Assume we are in the first case: the foliation is $p$-closed for almost every prime $\p$.   There exists  a complex manifold
$X_{\F}$ endowed with a holomorphic submersion $\pi : X_{\F} \to X$ and a section $\sigma : X \to X_{\F}$ such that the fiber of $\pi$ over $x \in X$ is the holonomy covering  of the leaf of $\F$ through $x$, see for instance \cite[Section 4.2]{Brunella}.

 Since $X$ is compact and the leaves of $\F$ are uniformized by Euclidean spaces according to Theorem \ref{TI:A}, it follows that $X_{\mathcal F}$ is Liouvillian in the sense of pluripotential theory: every  plurisbharmonic function bounded from above is constant. Moreover there exists a map $ \varphi: X_{\mathcal F} \to X\times X$ such that (a) $\varphi \circ \sigma(x)= (x,x)$; (b) the restriction of $\varphi$ to  $\pi^{-1}(x)$ is the holonomy covering
 of the leaf through $x$ of a copy of  $\mathcal F$ contained in $\{x\} \times X$. Therefore our foliation satisfies all the hypothesis of \cite[Theorem 2.2]{Bost} and we conclude that all the  leaves of $\F$  are  algebraic. Since $c_1(T\F) = c_2(T\F)=0$ then the leaves of $\F$ also have $c_1=c_2=0$ and, as recalled in the Introduction, it follows that they
 are all finite  coverings of Abelian varieties. The transverse foliation $\F^{\perp}$ defines isomorphisms between the distinct leaves and establishes the isotriviality of the family of Abelian varieties. Therefore after an unramified covering we arrive at the product of an Abelian variety $A$  with another  projective manifold $Y$ and the pull-back of $\F$ is given by the relative tangent sheaf of the projection $A \times Y \to Y$. This shows that assertion (1) in the statement of Theorem \ref{TI:D}  holds true.

 If instead we are in case (3) of Proposition \ref{P:saturation} then it is assertion (2)  in the statement of Theorem \ref{TI:D}  that holds true. The Theorem follows.

 The general case, where $R$ is a arbitrary finitely generated $\Z$-algebra, can be proved along the same lines.  If $\F$ is $p$-closed for every maximal prime $\p$ in a non-empty open subset of $Spec(R)$ and $K$, the field of fractions of $R$, have
positive transcendence degree over $\Q$ then we replace the $p$-closed foliation on the  projective manifold $X$ defined over $R$, by a family of $p$-closed foliations on projective manifolds over an affine base (with function field $K$) defined  over a number field. As the conditions on the Chern classes of $T\F$ are algebraic and stable under specialization, we can conclude that every foliation in the family (perhaps after restricting to a non-empty open subset of $B$) has Liouvillian leaves. As affine manifolds are also Liouvillian we can apply Bost's Theorem to this family in order to conclude. \qed

\subsection{A final remark}

We do believe that stable foliations on non-uniruled projective manifolds with $c_1(T\F)=0$ and $c_2(T\F)\neq 0$ have compact leaves.
Although the example presented in Section \ref{S:Faltings} tell us that case (3)  of Proposition \ref{P:saturation}
(the $p$-envelope $\G$ of a foliation with $c_1(T\F)=0$ has negative $c_1(T\G)$)  can happen,
we point out that   \cite[Theorem 5]{SB} implies that when $\dim(\F)=2$ and the $p$-envelope has negative first Chern class then $c_2(T\F)=0$.
A generalization of Shepherd-Barron' s
result to stable foliations with $c_1(T\F)=0$ and arbitrary dimension would leave open the possibility of using reduction modulo primes
to prove the compacteness of leaves.

\end{document}